\documentclass{article}

\usepackage{amssymb,amsmath,amsthm,url,bm,enumerate}
\usepackage{booktabs, multirow}
\usepackage{graphicx,color}
\usepackage{subfigure}

\usepackage{geometry}
\newtheorem{thm}{Theorem}[section]

\newtheorem{lem}[thm]{Lemma}
\newtheorem{cor}[thm]{Corollary}

\newtheorem{rema}[]{Remark}

\title{The smallest sets of points not determined by their X-rays\footnote{NOTICE: this is the authors version of a work that was accepted for publication in the \textbf{Bulletin of the London Mathematical Society}. Changes resulting from the publishing process, such as peer review, editing, corrections, structural formatting, and other quality control mechanisms may not be reflected in this document. Changes may have been made to this work since it was submitted for publication. A definitive version will subsequently appear in the Bulletin of the London Mathematical Society.}}
\author{Andreas Alpers\\
{\small Zentrum Mathematik, Technische Universit\"at M\"unchen} \\ {\small D-85747 Garching bei M\"unchen, Germany}\\
{\small \texttt{alpers@ma.tum.de}}\\
\and
David G. Larman\\
{\small Department of Mathematics, University College London}\\
{\small Gower Street, London WC1E 6BT, United Kingdom}\\
{\small \texttt{dgl@math.ucl.ac.uk}}\\
}

\date{}

\begin{document}
\maketitle

\begin{abstract}
Let $F$ be an $n$-point set in $\mathbb{K}^d$ with $\mathbb{K}\in\{\mathbb{R},\mathbb{Z}\}$ and $d\geq 2$. A (discrete) X-ray of $F$ in direction $s$ gives the number of points of $F$ on each line parallel to $s$. We define $\psi_{\mathbb{K}^d}(m)$ as the minimum number $n$ for which there exist $m$ directions $s_1,\dots,s_m$ (pairwise linearly independent and spanning $\mathbb{R}^d$) such that two $n$-point sets in $\mathbb{K}^d$ exist that have the same X-rays in these directions. The bound $\psi_{\mathbb{Z}^d}(m)\leq 2^{m-1}$ has been observed many times in the literature. In this note we show $\psi_{\mathbb{K}^d}(m)=O(m^{d+1+\varepsilon})$ for $\varepsilon>0$. For the cases $\mathbb{K}^d=\mathbb{Z}^d$ and $\mathbb{K}^d=\mathbb{R}^d$, $d>2$, this represents the first upper bound on $\psi_{\mathbb{K}^d}(m)$ that is polynomial in~$m$. As a corollary we derive bounds on the sizes of solutions to both the classical and two-dimensional Prouhet-Tarry-Escott problem. Additionally, we establish lower bounds on $\psi_{\mathbb{K}^d}$ that enable us to prove a strengthened version of R{\'e}nyi's theorem for points in $\mathbb{Z}^2$.
\end{abstract}

\section{Introduction}
The problem of reconstructing point sets from their X-rays has a long history; perhaps the 1952 paper~\cite{renyi52} by R{\'e}nyi represents one of the first works in this field. Of special interest are questions of uniqueness. Two sets with the same X-rays are said to be \emph{tomographically equivalent} \cite{gritzmannlangfeld}, \cite{hajdutijdeman01}; the sets are also commonly referred to as \emph{switching components} \cite{kongherman98}, \cite{shilferstein} or \emph{ghosts} \cite[Sect.~15.4]{herman09}. In \cite{matousek08} Matou{\v{s}}ek, P{\v{r}}{\'{\i}}v{\v{e}}tiv{\'y}, and {\v{S}}kovro{\v{n}} show that almost all sets of $m$~directions (in the sense of measure) allow for a unique reconstruction of $2^{Cm/\log(m)}$-point sets in the real plane (here $C>0$ is a constant and the result holds for large~$m$). For almost all choices of $m$ directions there thus exist only superpolynomial size switching components. By a careful selection of directions, however, we can reduce them to a polynomial size.

To make this precise, let $F$ be an $n$-point set in $\mathbb{K}^d$ with $\mathbb{K}\in\{\mathbb{R},\mathbb{Z}\}$ and $d\geq 2$. A (discrete) X-ray of~$F$ in direction $s$ gives the number of points of $F$ on each line parallel to $s$. We define $\psi_{\mathbb{K}^d}(m)$ as the minimum number $n$ for which there exist $m$ directions $s_1,\dots,s_m$ (pairwise linearly independent and spanning $\mathbb{R}^d$) such that two different $n$-point sets in $\mathbb{K}^d$ exist that have the same X-rays in these directions. We derive lower and upper bounds on $\psi_{\mathbb{K}^d}$.

Two constructions are known to yield upper bounds on $\psi_{\mathbb{K}^d}$. 
The first construction is based on regular polygons. The two disjoint $m$-point sets of alternate vertices of a regular $2m$-gon in $\mathbb{R}^2$ yield $\psi_{\mathbb{R}^2}(m)\leq m$. This cannot be transfered to $\mathbb{Z}^d$ as any (planar) regular polygon with integer vertices must have $3$, $4$ or $6$ vertices~\cite{scherrer,chrestenson63}. The functions $\psi_{\mathbb{R}^2}$ and $\psi_{\mathbb{Z}^2}$ are, in fact, different functions as we show $\psi_{\mathbb{Z}^2}(m)\geq m+1$ if $m=5$ or $m>6$ (see Thm.~\ref{thm:lowerbound}). From this we derive a strengthened version of R{\'e}nyi's theorem (see Thm.~\ref{thm:thm1} and Cor.~\ref{cor:cor1} in Sect.~\ref{sect:lowerbounds}).

The second well-known construction for upper bounds on $\psi_{\mathbb{K}^d}$ is based on two-colorings of the unit cube $[0,1]^m$ in $\mathbb{Z}^m$. More precisely, two different sets with equal X-rays in coordinate directions are obtained as the two disjoint sets of $2^{m-1}$ alternate vertices of  $[0,1]^m$. By projecting into~$\mathbb{Z}^d$, the bound $\psi_{\mathbb{Z}^d}(m)\leq 2^{m-1}$ is obtained. This construction seems to be due to Lorentz~\cite{lorentz49}; see also \cite{bianchilonginetti}, \cite[Lem.~2.3.2]{gardner}, and \cite[Thm.~4.3.1]{gardnergritzmannschapter4}. As $\mathbb{Z}^d\subseteq \mathbb{R}^d$, this, of course, yields also $\psi_{\mathbb{R}^d}(m)\leq 2^{m-1}$.
   
Our main observation is contained in the statement of Thm.~\ref{thm:maintheorem}, where we prove $\psi_{\mathbb{Z}^d}(m)=O(m^{d+1+\varepsilon})$ for $\varepsilon>0$. This is, to our knowledge, the first upper bound on $\psi_{\mathbb{K}^d}(m)$ that is polynomial in~$m$. Our proof is non-constructive.

We conclude in Sect.~\ref{sect:consequences} by stating some remarks and consequences that relate our bounds to the \emph{Prouhet-Tarry-Escott} problem from number theory (see, e.g., \cite[Sect.~21.9]{hardywright08}).

Throughout the paper, $\zeta$ is the Riemann zeta function, $m$ and $n$ denote natural numbers, and $\mathbb{Z}$, $\mathbb{R}$, $\mathbb{N}=\{1,2,\dots\}$ are, respectively, the sets of integers, reals, and natural numbers. We use the notation $\mathbb{N}_0=\mathbb{N}\cup\{0\}$, $[n]=\{1,\dots,n\}$, and $2\mathbb{Z}=\{2z\::\:z\in \mathbb{Z} \}$. With $G_n^d=[n]^d$ we denote the set of $d$-tuples of positive integers less than or equal to $n$. If $\xi\in\mathbb{R}$, then $\lceil \xi \rceil$ denotes the smallest integer greater than or equal to $\xi$. The symbol $O$ has the usual meaning: $f(m)=O(g(m))$ means that $f(m)/g(m)$ is bounded as $m \to \infty$. A property is said to hold for large $m$ if that property holds for all $m$ larger than some $m_0$.

\section{Lower bounds}\label{sect:lowerbounds}
In this section we derive lower bounds on $\psi_{\mathbb{K}^d}$. The key ideas are not new, but appear scattered and isolated in different contexts in the literature (see \cite{renyi52} and the proof of Thm.~\ref{thm:lowerbound} in \cite{alpers-tijdeman-07}). 

\begin{thm}\label{thm:thm1}
For every $d\geq2$ we have $\psi_{\mathbb{K}^d}(m)\geq m$.
\end{thm}
\begin{proof}
This is a reformulation of R{\'e}nyi's theorem (proved in \cite{renyi52} and generalized to arbitrary dimensions by Heppes~\cite{heppes}), which states that any $n$-point set in $\mathbb{K}^d$ is uniquely determined by its X-rays from $n+1$ different directions. 
For completeness, we reproduce a short proof. Suppose there are two sets $F, F'$ with equal X-rays in $m+1$ directions, each set containing at most $m$ points. Without loss of generality there exists a point $p \in F\setminus F'$. Since $F$ and $F'$ have equal X-rays, there needs to be a point of $F'$ on each of the $m+1$ lines through $p$. This implies that $F'$ contains at least $m+1$ points, a contradiction.
\end{proof}

The bound is tight for $m\in\{1,2,3,4,6\}$ and $\mathbb{K}^d=\mathbb{Z}^2$; examples showing this for $m=1,2,3,4,6$ are respectively provided by any two $1$-point sets in $\mathbb{Z}^2$,  two-colorings of the unit cube in $\mathbb{Z}^2$, the sets $F=\{(0,0),(1,2),(2,1)\}$, $F'=\{(1,0),(0,1),(2,2)\}$, and the examples shown in Fig.~4.3 and Fig.~4.5 of~\cite{gardnergritzmannschapter4}. For the remaining cases, however, we can improve the bound as stated in the following result.
\begin{thm} \label{thm:lowerbound}
If $m=5$ or $m>6$ then $\psi_{\mathbb{Z}^2}(m)\geq m+1$.
\end{thm}
\begin{proof}
Let $m=5$ or $m>6$, and suppose there exist different $n$-point sets $F,F' \subseteq \mathbb{Z}^2$ with equal X-rays in $m\geq n$ directions. Without loss of generality we can assume that $F \cap F'=\emptyset$. The convex hull $P$ of $F\cup F'$ is a non-degenerate polygon with at most $2n$ vertices.  Parallel to each of the $m$ directions there are two lines that support $P$ with each line containing a single point from $F$ and $F'$, respectively (since otherwise one of $F$ and $F'$ contains more than $n$ points). Since this implies that $P$ has at least $2m$ edges, we conclude that at least $2m$ of the elements of $F\cup F'$ are vertices of $P$ (i.e, $n=m$), proving that $F\cup F'$ is the set of vertices of the non-degenerate convex $2m$-gon $P$. Since $F$ and $F'$ have the same X-rays, $P$ has the property that any line through a vertex of $P$ in any of the $m$ directions meets another vertex of~$P$. Such polygons are known as lattice $U$-gons with $U$ denoting the set of $m$ directions. They, however, do not exist for $m >6$ (see Thm.~4.5 in \cite{gardnergritzmann97}). As is shown in the proof of Thm.~4.5 of \cite{gardnergritzmann97} or (more simply) in Thm.~6 of \cite{alpers-tijdeman-07}, there are also no lattice $U$-gons for exactly $5$ directions. In other words, we have $\psi_{\mathbb{Z}^2}(m)> m$ for $m=5$ or $m>6$.
\end{proof}

The bound is tight for $m=5$.  For this consider the $6$-point sets \[F=\{(0,2),(1,4),(2,2),(3,0),(4,3),(5,1)\} \textnormal{ and } F'=\{(0,3),(1,1),(2,4),(3,2),(4,0),(5,2)\}.\] It is easily verified that $F$ and $F'$ have the same X-rays in the $5$ directions \[S=\{(1,0),(0,1),(1,1),(1,-1), (-2,1)\}.\] 

A reformulation of Thm.~\ref{thm:lowerbound} provides a strengthened version of R{\'e}nyi's theorem for $\mathbb{Z}^2$.
\begin{cor}\label{cor:cor1}
Any $n$-point set in $\mathbb{Z}^2$ with $n=5$ or $n>6$ is uniquely determined by its X-rays taken from at least~$n$ different directions.
\end{cor}

\section{Upper bounds}\label{sect:upperbounds}
In this section we prove a polynomial upper bound on $\psi_{\mathbb{K}^d}$. As a prelude, we prove an upper bound on the number of lines parallel to a given direction that intersect points of $G_n^d$. This is followed by a lemma that asserts the existence of certain coverings of a specified finite part of the integer lattice by $m$ families of parallel lines. 

\begin{lem}\label{lem:lemq1}
For any relatively prime $d$-tuple $s=(\sigma_1,\dots,\sigma_d)\in\mathbb{N}_0^d\setminus\{0\}$ with $d\geq2$ there are at most $dn^{d-1}\cdot\max\{\sigma_1,\dots,\sigma_d\}$ lines parallel to $s$ that intersect~$G_n^d$.
\end{lem}
\begin{proof}
For each line $\ell$ parallel to $s=(\sigma_1,\dots,\sigma_d)$ that intersects $G_n^d$, there is a unique point $p \in \ell\cap G_n^d$ for which $p-s \not\in G_n^d$. The point $p-s$ needs to have a non-positive component, i.e., 
\[p\in V_i=\{(\xi_1,\dots,\xi_d)\in G_n^d\::\: 1\leq \xi_i\leq \sigma_i \}\] for an $i \in [d]$. As the number of points in $\bigcup_{i=1}^dV_i$ is clearly bounded by $dn^{d-1}\cdot\max\{\sigma_1,\dots,\sigma_d\}$, we obtain the claimed result. (Tight bounds can be obtained similary via the inclusion-exclusion principle, but they are not needed in the present context.)
\end{proof}


\begin{lem} \label{lem:lemT1}
Let $\varepsilon>0$, $m \in \mathbb{N}$, $d\geq 2$, and $n\in\left\{\left\lceil m^{1+(1+\varepsilon)/d}\right\rceil,\left\lceil m^{1+(1+\varepsilon)/d}\right\rceil+1\right\}$. Then, for large $m$ there is a set $S=\{s_1,\dots,s_m\}\subseteq \mathbb{Z}^d$ with the property that
\begin{enumerate}[(i)]
\item the elements of $S$ are pairwise linearly independent spanning $\mathbb{R}^d$;
\item the total number $l$ of lines that are parallel to a direction in $S$ and intersect $G_{n}^d$ is bounded from above by 
$2^{1+1/d} d {n}^{d-1} m^{1+1/d}.$
\end{enumerate}
\end{lem}
\begin{proof}
For the number $R(p,d)$ of relatively prime $d$-tuples in $G_p^d$, $p\in\mathbb{N}$, it holds by \cite{nymann72} that \[
\lim_{p \to \infty}\frac{R(p,d)}{p^d}=\frac{1}{\zeta(d)}.
\] As $\zeta$ decreases for values larger than $1$ and since $\zeta(2)=\pi^2/6<2$, we have \[ R(p,d)> p^d/2 \] for large $p$.

Setting $q=\left\lceil (2m)^{1/d}\right\rceil$, we note that $q \leq 2(2m)^{1/d}$ and $q \leq n$ for $m\geq2$. For large $m$ we have  
\[ R(q,d)> q^d/2\geq m, \] so for our set $S$ we can select $m$ elements from $G_{q}^d\subseteq G_{n}^d$. We can assume that the elements of $S$ span~$\mathbb{R}^d$ since otherwise we replace~$d$ of the directions by the standard unit vectors. Property~(i) is thus fulfilled (note that the elements of $S$ are relatively prime $d$-tuples).

The entries of the elements in $S$ are bounded by $q$, so by Lem.~\ref{lem:lemq1} we have at most
\[
mdn^{d-1}q\leq 2^{1+1/d} d {n}^{d-1} m^{1+1/d}
\]
lines parallel to a direction in $S$ that intersect $G_{n}^d$.
\end{proof}

\begin{thm} \label{thm:maintheorem}
For every $\varepsilon>0$ and $d\geq2$ it holds that $\psi_{\mathbb{Z}^d}(m)=O(m^{d+1+\varepsilon})$.
\end{thm}
\begin{proof}
We assume that $m$ is large enough that the set $S$ from Lem.~\ref{lem:lemT1} exists. We set \[n=\left\{\left\lceil m^{1+(1+\varepsilon)/d}\right\rceil,\left\lceil m^{1+(1+\varepsilon)/d}\right\rceil+1\right\}\cap 2\mathbb{Z}\] and $k=\frac{1}{2}n^d$. Note that $k \in \mathbb{N}$,  $k=O(m^{d+1+\varepsilon})$, and that we can assume that $n \geq 4$.

Let $l_i$, $i\in[m]$, denote the number of lines parallel to $s_i$ that intersect $G_n^d$. 
The X-ray in direction $s_i$ of a set in $G_{n}^d$ with cardinality $k$ gives a \emph{weak $k$-composition of $l_i$}, i.e., a  solution to $\xi_1+\cdots+\xi_{l_i}=k$ in nonnegative integers \cite[p.~15]{stanley}. (The converse is generally false, because the corresponding X-ray lines may intersect~$G_{n}^d$ in fewer points than provided by a weak $k$-composition of $l_i$.) 
The number of weak $k$-compositions of $l_i$ is given by \[N(k,l_i)= {k+l_i-1 \choose l_i-1}\] and thus represents an upper bound for the number of different X-rays of $k$-point subsets of $G_n^d$ in the direction $s_i$.

With $l=l_1+\cdots+l_m$ we thus obtain the following upper bound on the number of different X-rays (for the directions in $S$) that can originate from a subset of $G_{n}^d$ with cardinality $k$:
\[
 \prod_{i=1}^mN(k,l_i)\leq\prod_{i=1}^m{n^d/2+l_i\choose l_i} \leq \prod_{i=1}^m\left(\frac{(n^d/2+l_i)e}{l_i}\right)^{l_i}
=\prod_{i=1}^m\left( \frac{n^de}{2l_i}+e\right)^{l_i} \leq (ne+e)^{l}\leq n^{2l};
\]
here the inequalities (from left to right) follow from $N(k,l_i)\leq N(k,l_i+1)$, a standard inequality for binomial coefficients (see, e.g., \cite[Eq.~(4.9)]{odlyzko}), $l_i \geq n^{d-1}$, and  $n\geq4$, respectively.

There are 
\[
{n^d\choose n^d/2} \geq 2^{n^d/2}\label{eq:FF1}
\]
subsets of cardinality $k$ in $G_{n}^d$. We claim that 
\[
n^{2l} < 2^{n^d/2}
\] holds for large $m$, which, by the pigeonhole principle, concludes the proof as it implies the existence of two sets in $G_{n}^d$ with cardinality $k$ and equal X-rays in the directions in $S$.

For the claim we first note that 
\begin{equation}m^{1+(1+\varepsilon)/d} \leq n\leq 3m^{1+(1+\varepsilon)/d}   \label{eq:eqww1} \end{equation}
holds as $m^{1+(1+\varepsilon)/d}\geq 1$. It is easy to see that $\lim_{x \to \infty}x^{a}/2^{x^{b}}=0$ for $a,b>0$.
Thus for large~$m$ and $C=2^{3+1/d}d$
we have
\[ 3^{C}m^{C(1+(1+\varepsilon)/d)}<2^{m^{\varepsilon/d}},\] which, by~(\ref{eq:eqww1}) and Property~(ii) of Lem.~\ref{lem:lemT1}, gives
\[ n^{C}< 2^{m^{\varepsilon/d}}
\Rightarrow  n^{Cm^{1+1/d}}<2^n
\Rightarrow  n^{Cn^{d-1}m^{1+1/d}}<2^{n^d}
\Rightarrow  n^{4l}<2^{n^d},
\]
proving the claim.
\end{proof}


\section{Remarks and consequences}\label{sect:consequences}
The previously mentioned regular $2m$-gon construction in $\mathbb{R}^2$, together with the inequality $\psi_{\mathbb{R}^d}(m)\leq \psi_{\mathbb{Z}^d}(m)$ for $d\geq 2$, yields the following corollary to Thm.~\ref{thm:maintheorem}.
\begin{cor}
 For every $\varepsilon>0$ and $d\in\mathbb{N}$, it holds that 
 \[ \psi_{\mathbb{R}^d}(m)=\left\{ \begin{array}{lll} m &\textnormal{if}& d=2,\\ O(m^{d+1+\varepsilon}) &\textnormal{if}& d>2.\end{array}\right.\]
\end{cor}

In \cite{alpers-tijdeman-07} the  \emph{general Prouhet-Tarry-Escott problem} (PTE$_r$) was introduced: Given $k$, $n$, $r \in \mathbb{N}$, find two different multi-sets $\{x_1,\dots,x_n\}\!$, $\{y_1,\dots,y_n \}\subseteq\mathbb{Z}^r$ 
where $x_i=(\xi_{i1}, \dots, \xi_{ir})$, $y_i=(\eta_{i1}, \dots, \eta_{ir})$ for $i\in[n]$  such that
$$  \sum_{i=1}^n \xi_{i1}^{j_1}\xi_{i2}^{j_2} \cdots \xi_{ir}^{j_r}=\sum_{i=1}^n \eta_{i1}^{j_1}\eta_{i2}^{j_2} \cdots \eta_{ir}^{j_r} $$
for all nonnegative integers  $j_1, \dots, j_r$ with $j_1+\cdots+ j_r \leq k.$ The parameter $k$ is called the \emph{degree} and~$n$ the \emph{size} of the solution. 
Tracing back to works of Euler and Goldbach \cite[p.~705]{dickson}, the Prouhet-Tarry-Escott problem (PTE$_1$) is an old and largely unsolved problem in Diophantine analysis. The following corollary sharpens the bound of \cite[Thm.~12]{alpers-tijdeman-07} on the size of solutions, which for (PTE$_1$) is due to Prouhet~\cite{prouhet51}.

\begin{cor} \label{cor:cor2}
For every $\varepsilon>0$ there exists a constant $C>0$ such that there are solutions of (PTE$_2$) of degree $k$ and size bounded by $Ck^{3+\varepsilon}$.
\end{cor}
\begin{proof}
In Thm.~8 of \cite{alpers-tijdeman-07} it was shown that tomographically equivalent sets in $\mathbb{Z}^2$ for $m$ directions yield (PTE$_2$) solutions of degree $m-1$. This and Thm.~\ref{thm:maintheorem} for $d=2$ imply the statement of this corollary.
\end{proof}

\begin{rema}
As the products cancel, it is evident that solutions of (PTE$_1$) can be obtained by applying to (PTE$_2$) solutions a suitable linear functional that maps $(\xi_1,\xi_2)\in\mathbb{Z}^2$ to $\alpha_1\xi_1+\alpha_2\xi_2$
where $\alpha_1,\alpha_2 \in \mathbb{Z}$ are suitably chosen. The current best bounds for (PTE$_1$) are quadratic in~$k$ (see \cite{mezak61}, \cite{wright35}); the bound from Thm.~\ref{thm:maintheorem} is in this case weaker.
\end{rema}

\end{document}